\documentclass[11pt]{article}

\usepackage[T1]{fontenc}
\usepackage[utf8]{inputenc}
\usepackage{lmodern}
\usepackage[margin=1in]{geometry}
\usepackage{microtype}

\usepackage{amsmath,amssymb,amsthm,mathtools}
\usepackage{bm}

\usepackage{graphicx}
\graphicspath{{figures/}}
\usepackage{booktabs}
\usepackage{array}

\usepackage[shortlabels]{enumitem}
\usepackage[numbers,sort&compress]{natbib}
\usepackage{xcolor}
\usepackage[colorlinks=true,linkcolor=blue,citecolor=blue,urlcolor=blue]{hyperref}
\usepackage[capitalise,noabbrev]{cleveref}

\theoremstyle{plain}
\newtheorem{theorem}{Theorem}[section]
\newtheorem{lemma}[theorem]{Lemma}
\newtheorem{proposition}[theorem]{Proposition}
\newtheorem{corollary}[theorem]{Corollary}
\newtheorem{problem}[theorem]{Problem}
\crefname{problem}{Problem}{Problems}
\Crefname{problem}{Problem}{Problems}

\theoremstyle{definition}

\theoremstyle{remark}

\newcommand{\R}{\mathbb{R}}
\newcommand{\C}{\mathbb{C}}

\newcommand{\E}{\mathbb{E}}
\newcommand{\Pp}{\mathbb{P}}
\newcommand{\F}{\mathbb{F}}
\newcommand{\cD}{\mathcal{D}}
\newcommand{\cP}{\mathcal{P}}
\newcommand{\ind}[1]{\mathbf{1}\!\left\{#1\right\}}
\newcommand{\eas}{e.a.s.}
\DeclareMathOperator{\diag}{diag}
\DeclareMathOperator{\dist}{dist}
\DeclareMathOperator{\sgn}{sgn}

\title{Predicting Diagonalizability of a Mean Matrix}
\author{Jinze Zhao\\University of California, San Diego\\\texttt{jiz419@ucsd.edu}}
\date{}

\begin{document}

\maketitle

\begin{abstract}
Wu and Santhanam \cite{wu2024full} asked whether one can determine, from an increasing i.i.d. sample of binary random matrices, whether the unknown mean matrix is diagonalizable while making only finitely many errors almost surely. We answer this question affirmatively, for diagonalizability over either $\R$ or $\C$. The main observation is a general principle: every semialgebraic property of a fixed-dimensional bounded mean parameter is eventually almost surely predictable. We give a self-contained shrinking-confidence-set proof and an explicit predictor obtained from polynomial sign tests. Tarski--Seidenberg quantifier elimination shows that both the real- and complex-diagonalizable loci are semialgebraic, despite being neither closed nor open. We further extend the positive result to unbounded observations with any fixed finite moment of order $r>1$, using the Marcinkiewicz--Zygmund strong law. Combined with the Dembo--Peres topological criterion, this yields a sharp contrast: over the class of all merely integrable matrix laws, diagonalizability is not eventually almost surely predictable when the dimension is at least two. The construction is effective for fixed dimension, although no practical complexity bound is claimed.

\end{abstract}


\section{Introduction}
\label{sec:introduction}

Suppose that $X_1,X_2,\ldots$ are independent copies of a random $d\times d$ binary matrix $X$, and let
\[
    M=\E X
\]
denote the entrywise mean. At time $n$, a learner has seen only $X_1,\ldots,X_{n-1}$ and must predict a property of $M$. The learner succeeds in the eventual-almost-sure sense if, under every distribution in the model class, only finitely many of these predictions are wrong with probability one. This is the unsupervised prediction framework studied by \citet{wu2021prediction}.

Wu and Santhanam proved eventual-almost-sure predictability for several matrix properties, including singularity, rank, and repeated eigenvalues. They then posed the following question.

\begin{problem}[Wu--Santhanam, AISTATS 2021, Problem 1]
\label{prob:ws}
Is determining whether a matrix is diagonalizable \eas-predictable?
\end{problem}

The field of diagonalization is not specified in the original statement. Since $M$ is real, the two natural interpretations are similarity to a diagonal matrix over $\R$ and over $\C$. We resolve both.

The question is not an immediate application of the closed-set result in \citet{wu2021prediction}. The diagonalizable locus is neither closed nor open. For example, for $\varepsilon>0$,
\[
 A_\varepsilon=
 \begin{pmatrix}0&1\\0&\varepsilon\end{pmatrix}
 \longrightarrow
 \begin{pmatrix}0&1\\0&0\end{pmatrix},
\]
where every $A_\varepsilon$ has two distinct real eigenvalues but the limit is defective. Conversely,
\[
 \begin{pmatrix}0&t\\0&0\end{pmatrix}\longrightarrow 0
 \qquad (t\downarrow0),
\]
where the matrices on the left are defective and the limit is diagonalizable. Thus neither the property nor its complement can be handled as one closed hypothesis.

Our resolution uses the finite Boolean closure already implicit in the Wu--Santhanam framework, together with real algebraic geometry. The main contributions are as follows.

\begin{enumerate}[(i)]
    \item We prove that every semialgebraic property of a fixed-dimensional bounded mean vector is \eas-predictable. The proof gives a concrete shrinking-confidence-set predictor and is independent of the abstract prediction characterization.
    \item We prove that the loci of real matrices diagonalizable over $\R$ and over $\C$ are semialgebraic. The proof writes diagonalizability as an existential polynomial system and applies the Tarski--Seidenberg theorem.
    \item Combining these statements answers \cref{prob:ws} affirmatively. Dependence among entries of a single random matrix is allowed; only the matrices themselves are i.i.d.
    \item We extend the semialgebraic principle to matrix observations having any fixed finite $r$th moment with $r>1$. For $d\geq2$, the Dembo--Peres criterion shows that the corresponding assertion fails over the unrestricted class of laws having only a finite first moment. This yields a sharp moment dichotomy beyond the bounded model of the open problem.
\end{enumerate}

The underlying algebraic observation is short, but it applies well beyond diagonalizability: any fixed-dimensional mean property expressible by a first-order formula over the ordered real field is covered. Quantifier elimination makes the resulting predictor effective in principle; its worst-case computational cost may be prohibitive, and computational efficiency is not part of the claim.

\paragraph{Status of the question.}
The published AISTATS paper states \cref{prob:ws} in 2021, and the authors' substantially revised 2024 manuscript continues to list the same diagonalizability question as open \citep{wu2024full}. Searches by exact wording, title, citations, and the terms ``semialgebraic'' and ``Tarski--Seidenberg'' did not locate a prior resolution through August 11, 2026. This is evidence rather than a proof of novelty; author contact and database-level citation checks are recommended before submission.

\section{Model and main results}
\label{sec:model}

We first isolate the mean-prediction problem from its matrix specialization. Let $Y_1,Y_2,\ldots$ be i.i.d. random vectors in $\R^m$ with mean $\theta=\E Y_1$. For a parameter set $S\subseteq\R^m$, a predictor is a sequence of Borel maps
\[
    \Phi_N:(\R^m)^N\longrightarrow\{0,1\}.
\]
It is eventually almost surely correct for $S$ over a class of laws $\cP$ if, for every $P\in\cP$,
\begin{equation}
\label{eq:eas}
 P\!\left(
   \Phi_N(Y_1,\ldots,Y_N)=\ind{\theta\in S}
   \text{ for all sufficiently large }N
 \right)=1.
\end{equation}
Equivalently, the sum of the binary prediction errors is finite almost surely. The original protocol predicts at time $n$ from the first $n-1$ observations; replacing $N$ by $n-1$ gives exactly that convention.

Recall that a subset of $\R^m$ is \emph{semialgebraic} if it can be obtained from finitely many polynomial equalities and inequalities by finitely many unions, intersections, and complements. Equivalently, it is definable by a quantifier-free first-order formula over the ordered real field.

\begin{theorem}[Semialgebraic mean-property principle]
\label{thm:semialgebraic}
Let $m$ be fixed and let $S\subseteq\R^m$ be semialgebraic.
\begin{enumerate}[(a)]
    \item If the i.i.d. observations are supported on $[0,1]^m$, then membership of $\theta=\E Y_1$ in $S$ is \eas-predictable.
    \item Fix $r>1$. Over the class of all i.i.d. laws on $\R^m$ satisfying $\E\|Y_1\|_2^r<\infty$, membership of $\theta$ in $S$ is \eas-predictable.
\end{enumerate}
\end{theorem}

For the matrix statement, write $\cD_{\F,d}\subseteq\R^{d\times d}$ for the real matrices that are diagonalizable over $\F\in\{\R,\C\}$.

\begin{theorem}[Resolution of \cref{prob:ws}]
\label{thm:diagonalizable}
Fix $d\geq1$ and $\F\in\{\R,\C\}$. Let $X_1,X_2,\ldots$ be i.i.d. random $d\times d$ matrices supported on $[0,1]^{d\times d}$, and put $M=\E X_1$. There are Borel prediction rules $\Phi_N$, based on $X_1,\ldots,X_N$, such that for every underlying law,
\[
 \Pp\!\left(
   \Phi_N=\ind{M\in\cD_{\F,d}}
   \text{ for all sufficiently large }N
 \right)=1.
\]
In particular, this holds for the binary random matrices of \citet{wu2021prediction} and answers their Problem 1 affirmatively over either field.
\end{theorem}

The proof of \cref{thm:semialgebraic,thm:diagonalizable} is divided into two independent ingredients: \cref{sec:prediction} gives the statistical prediction lemma, while \cref{sec:algebraic} proves the required semialgebraicity.

\section{Predicting finite Boolean combinations of closed sets}
\label{sec:prediction}

The key statistical fact is elementary. A confidence region that eventually contains the true mean can test any fixed closed set. Finitely many such tests may then be combined without losing eventual correctness.

\subsection{A shrinking confidence ball}

Assume first that $Y_i\in[0,1]^m$ almost surely and set
\[
    \widehat\theta_N=\frac1N\sum_{i=1}^N Y_i,
    \qquad
    a_N=N^{-1/4},
    \qquad
    \rho_N=\sqrt m\,a_N.
\]

\begin{lemma}[Eventual confidence]
\label{lem:confidence}
For every law supported on $[0,1]^m$,
\[
 \Pp\!\left(
   \|\widehat\theta_N-\theta\|_2\leq\rho_N
   \text{ for all sufficiently large }N
 \right)=1.
\]
\end{lemma}

\begin{proof}
For each coordinate, Hoeffding's inequality \citep{hoeffding1963} gives
\[
 \Pp\!\left(
   |\widehat\theta_{N,j}-\theta_j|>a_N
 \right)
 \leq 2\exp(-2Na_N^2)
 =2\exp(-2\sqrt N).
\]
A union bound over the $m$ coordinates yields
\begin{equation}
\label{eq:hoeffding}
 \Pp\!\left(
   \|\widehat\theta_N-\theta\|_2>\rho_N
 \right)
 \leq 2m\exp(-2\sqrt N).
\end{equation}
The right-hand side is summable in $N$. The Borel--Cantelli lemma proves the claim.
\end{proof}

For a closed set $C\subseteq\R^m$, define
\begin{equation}
\label{eq:closed-predictor}
    T_{C,N}=\ind{\dist(\widehat\theta_N,C)\leq\rho_N},
\end{equation}
where $\dist(x,\varnothing)=+\infty$. Since distance to a nonempty closed set is continuous, $T_{C,N}$ is Borel measurable.

\begin{lemma}[Closed-set predictor]
\label{lem:closed}
For every closed $C\subseteq\R^m$,
\[
 T_{C,N}=\ind{\theta\in C}
 \qquad\text{for all sufficiently large $N$, almost surely.}
\]
\end{lemma}

\begin{proof}
Work on the probability-one event in \cref{lem:confidence}. If $\theta\in C$, then
\[
 \dist(\widehat\theta_N,C)
 \leq\|\widehat\theta_N-\theta\|_2
 \leq\rho_N
\]
eventually, so $T_{C,N}=1$. If $\theta\notin C$, closedness gives $\delta=\dist(\theta,C)>0$. Eventually $\rho_N<\delta/2$, and the reverse triangle inequality gives
\[
 \dist(\widehat\theta_N,C)
 \geq\delta-\|\widehat\theta_N-\theta\|_2
 \geq\delta-\rho_N
 >\rho_N.
\]
Thus $T_{C,N}=0$ eventually.
\end{proof}

\subsection{Finite Boolean closure}

Let $C_1,\ldots,C_k$ be closed subsets of $\R^m$ and let $h:\{0,1\}^k\to\{0,1\}$ be arbitrary. Consider
\begin{equation}
\label{eq:boolean-set}
 S_h=\left\{x\in\R^m:
 h\bigl(\ind{x\in C_1},\ldots,\ind{x\in C_k}\bigr)=1
 \right\}.
\end{equation}

\begin{proposition}[Finite Boolean prediction]
\label{prop:boolean}
The predictor
\[
 \Phi_N=h(T_{C_1,N},\ldots,T_{C_k,N})
\]
is eventually almost surely correct for $S_h$ over all laws supported on $[0,1]^m$.
\end{proposition}

\begin{proof}
By \cref{lem:closed}, each of the finitely many bits $T_{C_j,N}$ stabilizes almost surely to $\ind{\theta\in C_j}$. The intersection of these finitely many probability-one stabilization events still has probability one. Applying the fixed Boolean map $h$ proves the result.
\end{proof}

This proposition is a self-contained version, for bounded mean parameters, of the closed-set theorem and finite-Boolean-closure observation in \citet{wu2021prediction}. It also explains why a set need not be closed or open to be predictable.

\section{Real algebraic geometry of diagonalizability}
\label{sec:algebraic}

We now show that the relevant matrix property has exactly the finite Boolean structure required by \cref{prop:boolean}.

\begin{proposition}[Closed sign representation]
\label{prop:closed-sign}
Every semialgebraic set $S\subseteq\R^m$ is a finite Boolean combination of closed sets of the forms
\[
    \{x:q(x)=0\},\qquad
    \{x:q(x)\geq0\},\qquad
    \{x:q(x)\leq0\},
\]
where $q$ is a real polynomial.
\end{proposition}

\begin{proof}
By definition, $S$ has a quantifier-free formula involving finitely many polynomial sign conditions. Equalities and weak inequalities define closed sets. A strict condition $q>0$ is the complement of the closed set $\{q\leq0\}$, and $q<0$ is the complement of $\{q\geq0\}$. Substitution in the finite formula gives the claimed Boolean representation.
\end{proof}

We use the Tarski--Seidenberg projection theorem \citep{tarski1951,seidenberg1954,bochnak1998}: the image of a semialgebraic subset of $\R^{m+k}$ under a coordinate projection onto $\R^m$ is semialgebraic. Equivalently, existentially quantified systems of real polynomial equalities and inequalities admit quantifier elimination.

\begin{proposition}[Diagonalizable loci are semialgebraic]
\label{prop:diag-semialgebraic}
For every fixed $d$, both $\cD_{\R,d}$ and $\cD_{\C,d}$ are semialgebraic subsets of $\R^{d\times d}$.
\end{proposition}

\begin{proof}
For real diagonalizability, a real matrix $A$ belongs to $\cD_{\R,d}$ if and only if there exist $P\in\R^{d\times d}$ and $\lambda\in\R^d$ satisfying
\begin{equation}
\label{eq:real-certificate}
    AP=P\diag(\lambda),
    \qquad
    (\det P)^2>0.
\end{equation}
Every scalar entry of the matrix equation is a polynomial equality in the entries of $A,P,\lambda$, and the second condition is a strict polynomial inequality. Hence the set of triples $(A,P,\lambda)$ satisfying \cref{eq:real-certificate} is semialgebraic. Its projection onto the $A$-coordinates is exactly $\cD_{\R,d}$, so Tarski--Seidenberg proves the real claim.

For complex diagonalizability, write a prospective change-of-basis matrix and eigenvalue vector as
\[
    P=U+iV,
    \qquad
    \lambda=\alpha+i\beta,
\]
with $U,V\in\R^{d\times d}$ and $\alpha,\beta\in\R^d$. The equation $AP=P\diag(\lambda)$ is equivalent to the real polynomial matrix equations
\begin{align}
\label{eq:complex-real}
 AU&=U\diag(\alpha)-V\diag(\beta),\\
\label{eq:complex-imag}
 AV&=U\diag(\beta)+V\diag(\alpha).
\end{align}
Moreover, the real and imaginary parts of $\det(U+iV)$ are real polynomials in the entries of $U,V$, and invertibility is equivalent to
\begin{equation}
\label{eq:complex-det}
 \bigl(\Re\det(U+iV)\bigr)^2
 +\bigl(\Im\det(U+iV)\bigr)^2>0.
\end{equation}
Thus the tuples $(A,U,V,\alpha,\beta)$ satisfying \cref{eq:complex-real,eq:complex-imag,eq:complex-det} form a semialgebraic set whose projection onto $A$ is precisely $\cD_{\C,d}$. A second application of Tarski--Seidenberg completes the proof.
\end{proof}

\begin{proof}[Proof of \cref{thm:semialgebraic}(a)]
Apply \cref{prop:closed-sign} to write $S$ as a finite Boolean combination of closed polynomial sign sets, and then apply \cref{prop:boolean}.
\end{proof}

\begin{proof}[Proof of \cref{thm:diagonalizable}]
Vectorize the random matrices, so $m=d^2$ and $M$ is the bounded mean parameter. By \cref{prop:diag-semialgebraic}, the set $\cD_{\F,d}$ is semialgebraic for either $\F=\R$ or $\F=\C$. The claim follows from \cref{thm:semialgebraic}(a). Independence among the $d^2$ entries of one $X_i$ is never used: \cref{eq:hoeffding} requires only independence across the sample index $i$ for each fixed coordinate.
\end{proof}

\subsection{An explicit polynomial-sign predictor}
\label{sec:explicit}

The geometric predictor in \cref{eq:closed-predictor} is already explicit once a Boolean representation is known. There is also a version that avoids distance optimization. Quantifier elimination produces finitely many polynomials $q_1,\ldots,q_k$ and a Boolean function of their signs that represents $S$. On the cube $[0,1]^m$, let $L_j$ be any Lipschitz constant for $q_j$ with respect to $\|\cdot\|_\infty$. If
\[
 q_j(x)=\sum_\gamma c_{j,\gamma}x^\gamma,
\]
one valid choice is
\begin{equation}
\label{eq:lipschitz}
 L_j=\sum_\gamma |c_{j,\gamma}|\,|\gamma|_1.
\end{equation}
Define the estimated sign
\begin{equation}
\label{eq:sign-estimator}
 \widehat s_{j,N}=
 \begin{cases}
  +1,&q_j(\widehat\theta_N)> (L_j+1)a_N,\\
  -1,&q_j(\widehat\theta_N)<-(L_j+1)a_N,\\
  0,&|q_j(\widehat\theta_N)|\leq (L_j+1)a_N.
 \end{cases}
\end{equation}
On the event $\|\widehat\theta_N-\theta\|_\infty\leq a_N$, the Lipschitz inequality controls the polynomial error by $L_ja_N$. If $q_j(\theta)=0$, \cref{eq:sign-estimator} is eventually zero; if $q_j(\theta)\neq0$, the fixed nonzero margin dominates the shrinking threshold, and the correct strict sign is eventually returned. Hence every $\widehat s_{j,N}$ stabilizes to $\sgn q_j(\theta)$ almost surely. Evaluating the quantifier-free Boolean formula on these estimated signs gives another proof of \cref{thm:semialgebraic}(a).

Tarski's decision method makes this construction effective for fixed $d$ when the input polynomial systems have rational coefficients \citep{tarski1951,basu2006}. The claim is effectivity, not efficiency: general quantifier elimination and semialgebraic optimization can be very expensive as $d$ grows.

\section{Beyond bounded observations: a moment dichotomy}
\label{sec:moments}

The bounded Bernoulli model in \cref{prob:ws} is covered by Hoeffding concentration. The same algebraic argument survives for unbounded observations as soon as one has any fixed moment strictly above one.

\begin{lemma}[A deterministic mean rate under an $r$th moment]
\label{lem:mz-rate}
Fix $r>1$ and suppose $\E\|Y_1\|_2^r<\infty$. Choose
\[
 1<s<\min\{r,2\}
 \qquad\text{and}\qquad
 0<\gamma<1-\frac1s.
\]
Then
\[
 \|\widehat\theta_N-\theta\|_2\leq\sqrt m\,N^{-\gamma}
 \qquad\text{for all sufficiently large $N$, almost surely.}
\]
\end{lemma}

\begin{proof}
For each coordinate $j$, the Marcinkiewicz--Zygmund strong law \citep{marcinkiewicz1937} gives
\[
 N^{-1/s}\sum_{i=1}^N(Y_{i,j}-\theta_j)\longrightarrow0
 \qquad\text{almost surely},
\]
because the centered coordinate has a finite $s$th moment and $1<s<2$. Consequently,
\[
 |\widehat\theta_{N,j}-\theta_j|
 =o\!\left(N^{-(1-1/s)}\right)
 \leq N^{-\gamma}
\]
eventually almost surely. There are only finitely many coordinates, so the coordinatewise events hold simultaneously; the Euclidean bound follows.
\end{proof}

\begin{proof}[Proof of \cref{thm:semialgebraic}(b)]
Use the closed-set predictor \cref{eq:closed-predictor} with radius $\rho_N=\sqrt m\,N^{-\gamma}$, where $s$ and $\gamma$ are chosen as in \cref{lem:mz-rate}. The proof of \cref{lem:closed,prop:boolean} applies verbatim on the probability-one event in \cref{lem:mz-rate}. Finally use \cref{prop:closed-sign}.
\end{proof}

For completeness, we record the obstruction at the endpoint $r=1$. For $S\subseteq\R^m$, let $\cP_r(S)$ be the class of i.i.d. laws with finite $r$th moment and mean in $S$. A theorem of \citet{dembo1994} states, in particular, that for disjoint mean sets $A,B\subseteq\R^m$:
\begin{itemize}
    \item if $r>1$, the two classes $\cP_r(A)$ and $\cP_r(B)$ are \eas-distinguishable exactly when $A$ and $B$ are contained in disjoint $F_\sigma$ sets;
    \item if $r=1$, they are \eas-distinguishable exactly when $A$ and $B$ are contained in disjoint open sets.
\end{itemize}
Semialgebraic sets and their complements are $F_\sigma$, so the first part is consistent with, and more generally contextualizes, our direct positive proof.

\begin{corollary}[Sharp endpoint failure]
\label{cor:first-moment}
Fix $d\geq2$ and $\F\in\{\R,\C\}$. Over the class of all i.i.d. real $d\times d$ matrix laws satisfying only $\E\|X_1\|_2<\infty$, diagonalizability of $\E X_1$ over $\F$ is not \eas-predictable.
\end{corollary}

\begin{proof}
The mean parameter ranges over the connected space $\R^{d^2}$. Both $\cD_{\F,d}$ and its complement are nonempty: the zero matrix is diagonalizable, while a matrix containing the $2\times2$ Jordan block
\[
 J=\begin{pmatrix}0&1\\0&0\end{pmatrix}
\]
as a direct summand is not diagonalizable over either field. If $\cD_{\F,d}$ and its complement were contained in disjoint open sets $U,V$, then $U$ and $V$ would be nonempty and would cover $\R^{d^2}$, contradicting connectedness. The $r=1$ part of the Dembo--Peres criterion therefore rules out an \eas-predictor.
\end{proof}

There is no contradiction between \cref{thm:diagonalizable} and \cref{cor:first-moment}: bounded support supplies uniform tail control far stronger than mere integrability.

\section{The two-by-two case and interpretive limits}
\label{sec:example}

The smallest nontrivial dimension makes the Boolean structure transparent. Let
\[
 A=\begin{pmatrix}a&b\\c&d\end{pmatrix},
 \qquad
 \Delta(A)=(\operatorname{tr}A)^2-4\det A=(a-d)^2+4bc.
\]
Over $\C$, the matrix is diagonalizable if it has two distinct eigenvalues, or if its repeated-eigenvalue case is already scalar. Hence
\begin{equation}
\label{eq:2x2-complex}
 A\in\cD_{\C,2}
 \quad\Longleftrightarrow\quad
 \Delta(A)\neq0
 \ \text{or}\ 
 (a=d,\ b=0,\ c=0).
\end{equation}
Over $\R$, distinct eigenvalues must be real, so
\begin{equation}
\label{eq:2x2-real}
 A\in\cD_{\R,2}
 \quad\Longleftrightarrow\quad
 \Delta(A)>0
 \ \text{or}\ 
 (a=d,\ b=0,\ c=0).
\end{equation}
These formulas are already finite Boolean combinations of closed polynomial sign sets. In dimension two, no quantifier elimination is needed: the predictor of \cref{sec:explicit} can directly estimate the signs and zero status of $\Delta$, $a-d$, $b$, and $c$.

Several limitations are worth making explicit.
\begin{enumerate}[(i)]
    \item The dimension $d$ is fixed. The construction is not a dimension-uniform sample-complexity result.
    \item Eventual correctness is pointwise in the underlying law. The last error time may have no known deterministic bound, especially near the algebraic boundary.
    \item We establish predictability, not \eas-learnability in the stronger sense of a universally valid stopping rule certifying that no further errors will occur.
    \item Quantifier elimination provides a finite formula and a Borel predictor, but the general algorithm is not asserted to be computationally practical.
\end{enumerate}

The positive mechanism nevertheless applies unchanged to every fixed-dimensional property of a mean matrix expressible by a first-order formula with polynomial equalities and inequalities. Examples include prescribed rank conditions, existence of a real invariant subspace of fixed dimension, and polynomial Lyapunov-certificate properties. Each application still requires care about the intended field and quantifiers, but no new statistical argument.

\section{Conclusion}
\label{sec:conclusion}

Diagonalizability of an unknown mean matrix is eventually almost surely predictable in the Wu--Santhanam binary-matrix model. The proof is a synthesis of two facts: semialgebraic mean properties are finite Boolean combinations of closed tests, and diagonalizability over either $\R$ or $\C$ is semialgebraic by Tarski--Seidenberg projection. A shrinking confidence region makes every closed test eventually correct, including at its boundary, and finitely many stabilized bits can be combined without error.

The argument yields more than the requested resolution. It covers all fixed-dimensional first-order real-algebraic properties of bounded mean parameters and remains valid under any fixed moment assumption of order strictly greater than one. In contrast, the Dembo--Peres criterion rules out diagonalizability prediction over the unrestricted class of merely integrable matrix laws in dimension at least two. The remaining questions are primarily quantitative and computational: finding small quantifier-free certificates for structured matrix properties, deriving useful finite-sample error profiles away from algebraic boundaries, and identifying subclasses for which the predictor can be implemented efficiently.

\section{Disclosure}
\label{sec:disclosure}
The proof strategy and counterexample were produced by OpenAI’s GPT-5.6 Sol Ultra through Codex in
response to prompts from the author. Codex was also used to revise
the exposition and prepare the LaTeX manuscript. The author selected the
problem, directed the interactions and revisions, and is the sole named author. The AI system is acknowledged as a reasoning and writing tool, not
as an author. This disclosure is not a substitute for independent expert
mathematical review.
\bibliographystyle{plainnat}
\bibliography{references}

\appendix
\section{Alternative algebraic certificates}
\label{app:certificates}

The change-of-basis certificates in \cref{prop:diag-semialgebraic} are direct and work uniformly over $\R$ and $\C$. A second description uses annihilating polynomials and may be useful for symbolic implementations.

\begin{proposition}
\label{prop:minpoly}
A real matrix $A$ is diagonalizable over $\C$ if and only if there exist an integer $1\leq k\leq d$ and a monic real polynomial
\[
 q(t)=t^k+c_{k-1}t^{k-1}+\cdots+c_0
\]
such that $q(A)=0$ and $q$ is square-free. Equivalently, the resultant $\operatorname{Res}(q,q')$ is nonzero.
\end{proposition}

\begin{proof}
If $A$ is diagonalizable over $\C$, its minimal polynomial is the product of the distinct linear factors corresponding to its distinct eigenvalues. Because $A$ is real, this polynomial has real coefficients; it is square-free and annihilates $A$. Conversely, if a square-free polynomial $q$ annihilates $A$, then the minimal polynomial of $A$ divides $q$ and is itself square-free. A complex matrix is diagonalizable exactly when its minimal polynomial is square-free. The resultant condition is the standard polynomial certificate that $q$ and $q'$ have no common root.
\end{proof}

For each fixed $k$, the equations $q(A)=0$ and the inequality $\operatorname{Res}(q,q')^2>0$ are polynomial conditions in $A,c_0,\ldots,c_{k-1}$. Taking a finite union over $k$ and projecting out the coefficients gives another Tarski--Seidenberg proof for $\cD_{\C,d}$.

Over $\R$, one may instead require distinct real numbers $\lambda_1,\ldots,\lambda_k$ such that
\[
 \prod_{j=1}^k(A-\lambda_jI)=0,
 \qquad
 \prod_{1\leq i<j\leq k}(\lambda_i-\lambda_j)^2>0.
\]
The minimal polynomial then splits into distinct real linear factors, which is equivalent to real diagonalizability.

\section{Timing convention and measurability}
\label{app:timing}

The predictor in the main text uses $N$ observations and labels its output $\Phi_N$. In the protocol of \citet{wu2021prediction}, the prediction at time $n$ is made before $X_n$ is revealed and may use $X_1,\ldots,X_{n-1}$. Define their time-$n$ rule to be our $\Phi_{n-1}$ for $n\geq2$, with an arbitrary initial output at $n=1$. This changes at most one error and therefore preserves \eas-predictability.

Distance to a nonempty closed subset of Euclidean space is continuous. Thus the tests in \cref{eq:closed-predictor} are Borel. The explicit sign tests in \cref{eq:sign-estimator} are also Borel because polynomials are continuous. No measurable-selection argument is required.

\section{A compact proof through the original framework}
\label{app:compact}

There is a short alternative proof of the bounded result using only established components of \citet{wu2021prediction} and the algebraic proposition above. Their Theorem 6 states that a closed subset of the bounded mean-parameter cube is \eas-predictable. Immediately afterward, they observe that applying any fixed Boolean function to finitely many \eas-predictable indicators preserves \eas-predictability. By \cref{prop:closed-sign}, every semialgebraic set is a finite Boolean combination of closed polynomial sign sets. By \cref{prop:diag-semialgebraic}, diagonalizability over either field is semialgebraic. The result follows. Sections~\ref{sec:prediction} and \ref{sec:explicit} supply a direct predictor and verify its boundary behavior rather than treating the closed-set theorem as a black box.

\end{document}